\documentclass[reqno]{amsart}
\usepackage{amssymb,amsmath}
\usepackage{amsthm}
\usepackage{color,graphicx}
\usepackage{eucal}

\newcommand{\R}{\mathbb R}

\newcommand{\tres}{|\!|\!|}

\newcommand{\peq}{\hspace*{0.10in}}

\newcommand{\quebra}{\hspace*{-0.25in}}

\newtheorem{theorem}{Theorem}[section]
\newtheorem{proposition}[theorem]{Proposition}
\newtheorem{remark}[theorem]{Remark}
\newtheorem{lemma}[theorem]{Lemma}

\begin{document}
\vglue-1cm \hskip1cm
\title[Generalized KdV equation]{The supercritical generalized KdV equation: Global
well-posedness in the energy space and below }



\author[L. G. Farah]{Luiz G. Farah}
\address{ICEx, Universidade Federal de Minas Gerais, Av. Ant\^onio Carlos, 6627, Caixa Postal 702, 30123-970,
Belo Horizonte-MG, Brazil}
\email{lgfarah@gmail.com}

\author[F. Linares]{Felipe Linares}
\address{IMPA, Estrada Dona Castorina 110, CEP 22460-320, Rio de Janeiro, RJ,
 Brazil.}
\email{linares@impa.br}
\thanks{The second author was partially supported by CNPq/Brazil.}

\author[A. Pastor]{Ademir Pastor}
\address{IMECC-UNICAMP, Rua S\'ergio Buarque de Holanda, 651, 13083-859, Cam\-pi\-nas-SP, Bra\-zil}
\email{apastor@ime.unicamp.br}




 \begin{abstract}
 We consider the generalized Korteweg-de Vries (gKdV) equation
 $\partial_t u+\partial_x^3u+\mu\partial_x(u^{k+1})=0$,
where $k\geq5$ is an integer number and $\mu=\pm1$. In the focusing
case ($\mu=1$), we show that if the initial data $u_0$ belongs to
$H^1(\R)$ and satisfies $E(u_0)^{s_k} M(u_0)^{1-s_k} < E(Q)^{s_k}
M(Q)^{1-s_k}$, $E(u_0)\geq0$, and $\|\partial_x
u_0\|_{L^2}^{s_k}\|u_0\|_{L^2}^{1-s_k} < \|\partial_x
Q\|_{L^2}^{s_k}\|Q\|_{L^2}^{1-s_k}$, where $M(u)$ and $E(u)$ are the
mass and energy, then the corresponding solution is global in
$H^1(\R)$. Here, $s_k=\frac{(k-4)}{2k}$ and $Q$ is the ground state
solution corresponding to the gKdV equation. In the defocusing case
($\mu=-1$), if $k$ is even, we prove that the Cauchy problem is
globally well-posed in the Sobolev spaces $H^s(\mathbb{R})$,
$s>\frac{4(k-1)}{5k}$.
 \end{abstract}

\maketitle

\section{Introduction}\label{introduction}

Consider the Initial Value Problem (IVP) associated with the supercritical  generalized Korteweg-de Vries
(gKdV) equation, i.e.,
\begin{equation}\label{gkdv}
\begin{cases}
\partial_t u+\partial_x^3u+\mu\partial_x(u^{k+1})=0, \;\;x\in\R, \;t>0, \\
u(x,0)=u_0(x),
\end{cases}
\end{equation}
where $\mu=\pm1$.

Local well-posedness of the Cauchy problem \eqref{gkdv} (with $k\geq1$) has been studied by many
authors in recent years. We refer the reader to Kenig, Ponce and Vega \cite{kpv1}, \cite{kpv2} for
a complete set of sharp results.

Our main interest here is on global well-posedness. Let us briefly recall the best
results available in the literature. For $k=1$ and $k=2$, global well-posedness was established by
Colliander, Keel, Staffilani, Takaoka, and Tao \cite{CKSTT3} for data, respectively, in $H^s(\mathbb{R}),
s>-3/4$ and $H^s(\mathbb{R}), s>1/4$, and by Guo \cite{Gu} for data, respectively, in $H^{-3/4}(\mathbb{R})$
and $H^{1/4}(\mathbb{R})$. These results show to be sharp in view of the work of   Kenig, Ponce, and Vega
\cite{KPV5} (see also \cite{BKPSV96}, \cite{CCT},  \cite{Tz}).

The case $k=3$ was dealt with by Gr\"unrock, Panthee, and Silva
\cite{GPS}, where the authors showed global well-posedness in
$H^s(\mathbb{R}), s>-1/42$. It should be pointed out that for $k=3$,
Tao \cite{tao} established a local existence result in
$\dot{H}^{-\frac16}(\R)$, the critical (scale-invariant) space,
therefore for small data the solutions extend globally. For recent
progress in this case we refer Koch and Marzuola \cite{koch-m}.
Under ``sharp smallness condition'', the critical case $k=4$ was
studied by Fonseca, Linares, and Ponce in \cite{FLP1}. There it was
established global well-posedness in $H^s(\mathbb{R}), s>3/4$. Farah
\cite{LG6} used the I-method of \cite{CKSTT3}, to further lower the
regularity of the initial data to $s > 3/5$. Recently, Miao, Shao,
Wu, and Xu \cite{miao}, improved the latter result to initial data
in $H^s(\R)$, $s>6/13$. Their method of proof combines the I-method
with a multilinear correction analysis.  For $k=4$, Kenig, Ponce and
Vega \cite{kpv1} showed local well-posedness for data in $L^2$ the
critical space in this case which for small data yield global
solutions. Finally, we should mention that for $k=4$, Merle
\cite{mer} and Martel and Merle \cite{marmer} proved the existence
of real-valued solutions of \eqref{gkdv} in $H^1(\R)$ corresponding
to data in $u_0\in H^1(\R)$ with $\|u_0\|_{L^2}>\|Q\|_{L^2}$ that
blow-up. For $k>4$ it is an outstanding open problem.

As far as we are concerned, for $k\geq5$, no global results below the energy space are available.
Not even a precise description of the conditions to obtain $H^1$ global solutions. These facts
motivate the present study.

To start with the local results, using a scaling argument let us motivate what should be the Sobolev spaces to
studying \eqref{gkdv}. Note if $u$ is a solution of \eqref{gkdv}, then, for any $\lambda>0$,
$u_\lambda(x,t)=\lambda^{2/k}u(\lambda x,\lambda^3t)$ is also a solution with initial data
$u_\lambda(x,0)=\lambda^{2/k}u_0(\lambda x)$. Moreover,
$$
\|u_\lambda(\cdot,0)\|_{\dot{H}^s}=\lambda^{s+2/k-1/2}\|u_0\|_{\dot{H}^s}.
$$
Thus, for each $k$ fixed, the scale-invariant Sobolev space is $\dot{H}^{s_k}$, $s_k=1/2-2/k$. Therefore, the natural
Sobolev spaces to studying \eqref{gkdv} are  ${H}^{s}$, $s>s_k=1/2-2/k$. Actually, this question has already been
addressed by Kenig, Ponce, and Vega \cite{kpv1}. More precisely, they show the following.

\begin{theorem}\label{local}
Let $k>4$ and $s>s_k=(k-4)/2k$. Then for any $u_0\in H^s(\R)$ there exist $T=T(\|u_0\|_{H^s})>0$
(with $T(\rho;s)\to 0$ as $\rho\to 0$) and a unique strong solution $u(\cdot)$ of the IVP \eqref{gkdv}
satisfying:
\begin{equation}  \label{a1}
u\in C([-T,T]:H^s(\R)),
\end{equation}
\begin{equation}\label{a2}
\|u\|_{L^5_xL^{10}_T}+\|D^{s}_xu\|_{L^5_xL^{10}_T}<\infty
\end{equation}
\begin{equation}\label{a3}
\|u_x\|_{L^\infty_xL^{2}_T}+\|D^{s}_xu_x\|_{L^\infty_xL^{2}_T}<\infty,
\end{equation}
and
\begin{equation}\label{a3b}
\|D^{\gamma_k}_tD^{\alpha_k}_xD^{\beta_k}_tu\|_{L^{p_k}_xL^{q_k}_T}<\infty
\end{equation}
where
\begin{equation} \label{a4}
\alpha_k=\frac{1}{10}-\frac{2}{5k}, \qquad
\beta_k=\frac{3}{10}-\frac{6}{5k}, \qquad
\gamma_k=\gamma_k(s)=\frac{s-s_k}{3}
\end{equation}
\begin{equation} \label{a5}
\frac{1}{p_k}=\frac{2}{5k}+\frac{1}{10}, \qquad \frac{1}{q_k}=\frac{3}{10}-\frac{4}{5k}.
\end{equation}
Furthermore, given $T'\in(0,T)$ there exists a neighborhood $V$ of $u_0$ in $H^s(\R)$ such that the map
$u_0\mapsto \tilde{u}(t)$ from $V$ into the class defined by \eqref{a1}-\eqref{a3} with $T'$ instead of $T$
is smooth.
\end{theorem}

The method to prove Theorem \ref{local} combines smoothing effects and Strichartz-type estimates together with
the Banach contraction principle. As a matter of fact, the original theorem stated in \cite{kpv1} differs slightly
in the functional spaces. Here, we will give a skech of the proof of
Theorem \ref{local} in this functional spaces setting.

\begin{remark} It should be observed that in \cite{kpv1} the authors also showed a local result for initial data in
$\dot{H}^{s_k}(\R)$, $s_k$ as above, but $T=T(u_0)$, that is, the existence time $T$ depends on $u_0$ itself and not
on $\|u_0\|_{\dot{H}^{s_k}}$ (see also \cite{BKPSV96}) and that this is global if $\|u_0\|_{\dot{H}^{s_k}}\le c_k$
for real or complex-valued data.
\end{remark}

Once Theorem \ref{local} is established, a natural question presents itself: can the real solutions be extended
globally-in-time? Such a question has mathematical and physical interest and it has  been widely studied in the
past few years.

By observing that the flow of the gKdV equation is conserved by the quantities:
\begin{equation}\label{MC}
Mass\equiv M(u(t))=\int u^2(t)\, dx
\end{equation}
and
\begin{equation}\label{EC}
Energy\equiv E(u(t))=\frac{1}{2}\int(\partial_xu)^2(t)\;dx-\frac{\mu}{k+2}\int u^{k+2}(t)\;dx,\\
\end{equation}
one can partially answer this question for solutions in $H^1(\R)$ if the initial data is small. Indeed, the
quantities $M$ and $E$ allows us to obtain a priori estimates as follow: Using $M$ we can control $\|u(t)\|_{L^2}$.
In order to control $\|\partial_xu(t)\|_{L^2}$, we
use $E$ by writing
\begin{equation}\label{ap1}
\|\partial_x u(t)\|_{L^2}^2=2E(u_0)+\frac{2\mu}{k+2}\int u^{k+2}(t)\;dx.
\end{equation}
The Gagliardo-Nirenberg inequality yields
\begin{equation}\label{ap2}
\begin{split}
\int u^{k+2}(t)\;dx&\le c\,\|\partial_x u(t)\|_{L^2}^{\theta(k+2)}\|u(t)\|_{L^2}^{(1-\theta)(k+2)},
\quad\theta=\frac{k}{2(k+2)},\\
&=c\,\|u_0\|_{L^2}^{(k+4)/2}\|\partial_x u(t)\|_{L^2}^{k/2}.
\end{split}
\end{equation}
From \eqref{ap1} and \eqref{ap2} it follows that
\begin{equation}\label{ap3}
\|\partial_x u(t)\|_{L^2}^2\le 2E(u_0)+\frac{2c}{k+2}\,\|u_0\|_{L^2}^{(k+4)/2}\|\partial_x u(t)\|_{L^2}^{k/2}.
\end{equation}
Now let $X(t)=\|\partial_x u(t)\|_{L^2}^2$, for $t\in (0,T)$ ($T$ given by Theorem \ref{local}). Since $k>4$ the
inequality \eqref{ap3} can be written as
\begin{equation}\label{ap4}
X(t)\le 2E(u_0)+\frac{2c}{k+2}\,\|u_0\|_{L^2}^{(k+4)/2}X(t)^{1+\epsilon},\quad \epsilon>0,
\end{equation}
or
\begin{equation}\label{ap5}
X(t)-c(k,\|u_0\|_{L^2})\,X(t)^{1+\epsilon}\le 2E(u_0).
\end{equation}
Thus if $0\le 2E(u_0)$ is not too large one can guarantee the existence of $0<\beta_1<\beta_2$ where
the inequality \eqref{ap5} holds in the intervals $[0,\beta_1]$ and $[\beta_2,\infty)$. By continuity if
we have $X(0)\in [0,\beta_1]$, $X(t)$ will remain there for $t\in (0,T)$. Hence $\|\partial_x u(t)\|_{L^2}$
will be bound for $t\in (0,T)$ and we can apply the local result to extend the solution. The argument
works if the initial data is small enough, i.e., $\|u_0\|_{H^1}$ is sufficiently small to satisfy the conditions
along the previous argument.

Note that the case where $\mu=-1$ and $k$ even is, in some sense, special. Indeed, since $k$ is even we have
$\int u^{k+2}(x,t)dx>0$, for all $t>0$, which implies
\begin{equation*}
\|\partial_xu(t)\|^2_{L^2}\lesssim E(u)(t).
\end{equation*}
Therefore, we have an \emph{a priori} bound for $\|\partial_x u(t)\|_{L^2}$ which, together with mass conservation
\eqref{MC} and local theory, implies global well-posedness without any smallness condition.

The above discussion can be summarized in the following theorem (see also \cite[Theorem 2.15]{kpv1}).

\begin{theorem}\label{global0}
Let $k>4$ and $s>s_k=(k-4)/2k$. Then
\begin{itemize}
\item [(a)] if $\mu=\pm 1$, there exists $\delta_k>0$ such that for any $u_0 \in H^1$ with
$$\|u_0\|_{H^1}<\delta_k$$
there exists a unique strong solution $u(\cdot)$ of the IVP \eqref{gkdv} satisfying
\begin{equation*}
u\in C(\R:H^1(\R)) \cap L^{\infty}(\R:H^1(\R)).
\end{equation*}
\item [(b)] if $\mu=-1$ and $k$ is even then the same statement is true without any smallness assumption on the
initial data.
\end{itemize}
\end{theorem}

We have two main goals in this paper. The first one is to make precise the $H^1$-size of the initial data
(in the preceding argument) to construct global $H^1$ solutions when $\mu=1$ or $\mu=-1$ and $k$ odd. The second one
is to loosen the regularity requirements on the initial data which ensure global-in-time solutions for the
IVP (\ref{gkdv}) when $\mu=-1$ and $k$ even. Below we also explain why we cannot apply the same method when
$\mu=1$ or $\mu=-1$ and $k$ odd (see Remark \ref{KODD}).

We consider first the focusing case $\mu=1$ or the defocusing case $\mu=-1$ with $k$ odd.
As explained above, it is not clear how large is the size of the initial data in $H^1$ to obtain global solutions.
The next theorem shows us how small the initial data should be.

\begin{theorem}\label{global3}
 Let $u_0\in H^1(\R)$. Let $k>4$ and $s_k=(k-4)/2k$. Suppose that
\begin{equation}\label{GR1}
E(u_0)^{s_k} M(u_0)^{1-s_k} < E(Q)^{s_k} M(Q)^{1-s_k} , \,\,\, E(u_0) \geq 0.
\end{equation}
If
\begin{equation}\label{GR2}
\|\partial_x u_0\|_{L^2}^{s_k}\|u_0\|_{L^2}^{1-s_k} < \|\partial_x Q\|_{L^2}^{s_k}\|Q\|_{L^2}^{1-s_k},
\end{equation}
then for any $t$ as long as the solution exists,
\begin{equation}\label{GR3}
\|\partial_x u(t)\|_{L^2}^{s_k}\|u_0\|_{L^2}^{1-s_k}=\|\partial_x u(t)\|_{L^2}^{s_k}\|u(t)\|_{L^2}^{1-s_k}
< \|\partial_x Q\|_{L^2}^{s_k}\|Q\|_{L^2}^{1-s_k},
\end{equation}
where $Q$ is unique positive radial solution of the elliptic equation
\begin{equation*}
\Delta Q-Q+Q^{k+1}=0.
\end{equation*}
This in turn implies that $H^1$ solutions exist globally in time.
\end{theorem}

To prove Theorem \ref{global3}, we follow closely the arguments in  Holmer and Roudenko \cite{HR08} which were
inspired by those introduced by Kenig and Merle \cite{kenig-merle}.

Next we consider the defocusing case $\mu=-1$ with $k$ even. Our main result is the following.

\begin{theorem}\label{T1}
Let $\mu=-1$ and assume that $k$ is even. Let $u_0\in H^s(\mathbb{R}), s>\frac{4(k-1)}{5k}$. Then, the local
solution in Theorem \ref{local} can be extended to any time interval. Moreover, for all $T>0$, the solution satisfies
\begin{equation}\label{pb}
\sup_{t\in[0,T]}\left\{\|u(t)\|^2_{H^{s}}\right\}\leq C(1+T)^{\frac{(1+4/k)(1-s)}{5s-4(k-1)/k}+},
\end{equation}
where the constant $C$ depends only on $s$ and $\|u_0\|_{H^{s}}$.
\end{theorem}

\begin{remark}
Note that when $k=4$ we recover the result proved in Farah \cite{LG6}.
\end{remark}

Here we use the approach introduced by Colliander, Keel, Staffilani, Takaoka and Tao in \cite{CKSTT4}, the so-called
$I$-method. We also explain why the refined approach introduced by the same authors in \cite{CKSTT3},
cannot be used to improve our global result stated in Theorem \ref{T1} (see Proposition \ref{P1}).

Note that when $u_0 \in H^{s}(\R)$ with
$s<1$ in (\ref{gkdv}), the energy \eqref{EC} could be infinite, and so the
conservation law (\ref{EC}) is meaningless. To overcome this
difficulty, by following the $I$-method scheme, we introduce a modified
energy functional which is also defined for less regular functions.
Unfortunately, this new functional is not strictly conserved, but we
can show that it is \textit{almost} conserved in time. When one is
able to control its growth in time explicitly this allows to iterate
a modified local existence theorem to continue the solution to any
time $T$.

The plan of this paper is as follows. In the next section we introduce some notation and preliminaries.
In Section \ref{localwp} we prove Theorem \ref{local}. Next, in Section \ref{mu=1}, we show Theorem \ref{global3}.
The result of global well-posedness in Theorem \ref{T1} is proved in Section \ref{impglobal}.


\section{Notation and Preliminaries}\label{notation}

Let us start this section by introducing the notation used throughout the paper. We use $c$ to denote various
constants that may vary line by line. Given any positive numbers $a$ and $b$, the notation $a \lesssim b$ means
that there exists a positive constant $c$ such that $a \leq cb$. Also, we denote $a \sim b$ when, $a \lesssim b$
and $b \lesssim a$. We use $a+$ and $a-$ to denote $a+\varepsilon$ and $a-\varepsilon$, respectively, for
arbitrarily small $\varepsilon>0$.

We use $\|\cdot\|_{L^p}$ to denote the $L^p(\R)$ norm. If necessary, we use subscript to inform which variable
we are concerned with. The mixed norm $L^q_tL^r_x$ of $f=f(x,t)$ is defined as
\begin{equation*}
\|f\|_{L^q_tL^r_x}= \left(\int \|f(\cdot,t)\|_{L^r_x}^q dt \right)^{1/q},
\end{equation*}
with the usual modifications when $q =\infty$ or $r=\infty$. The  $L^r_xL^q_t$ norm is similarly defined.

We define the spatial Fourier transform of $f(x)$ by
\begin{equation*}
\hat{f}(\xi)=\int_{\R}e^{-ix\xi}f(x)dx
\end{equation*}
and the space-time Fourier transform of $u(x,t)$ by
\begin{equation*}
\widetilde{u}(\xi, \tau)=\int_{\R}\int_{\R}e^{-i(x\xi+t\tau)}u(x,t)dtdx.
\end{equation*}
Note that the derivative $\partial_x$ is conjugated to multiplication by $i\xi$ by the Fourier transform.

The set of Schwartz functions is represented by $\mathcal{S}(\R)$. We shall also define $D^s$ and $J^s$ to be,
respectively, the Fourier multipliers with symbols $|\xi|^s$ and $\langle \xi \rangle^s = (1+|\xi|)^s$. Thus, the
norm in the Sobolev space $H^s(\R)$ is given by
\begin{equation*}
\|f\|_{H^s}\equiv \|J^sf\|_{L^2_x}=\|\langle \xi \rangle^s\hat{f}\|_{L^2_{\xi}}.
\end{equation*}

We also define the spaces $X_{s,b}(\R\times\R)$  on $\R\times\R$ through the norm
\begin{equation*}
\|F\|_{X_{s,b}(\R\times\R)}=\|\langle\tau-\xi^3\rangle^b\langle\xi\rangle^s \widetilde{F}\|_{L^{2}_{\xi,\tau}}.
\end{equation*}
These spaces were introduced in the study of nonlinear dispersive wave problems
by Bourgain \cite{B}.

For any interval $I$, we define the localized $X_{s,b}(I \times\R)$ spaces by
\begin{equation*}
\|u\|_{X_{s,b}(I \times\R)}=\inf\left\{\|w\|_{X_{s,b}(\R \times\R)}:w(t)=u(t) \textrm{ on }  I\right\}.
\end{equation*}
We often abbreviate $\|u\|_{X_{s\!,b}}$ and
$\|u\|_{X^{I}_{s\!,b}}$, respectively, for
$\|u\|_{X_{s\!,b}(\R\times\R)}$ and
$\|u\|_{X_{s\!,b}(I\times\R)}$.

Let us introduce now some useful  lemmas and inequalities. In what follows, $U(t)$ denotes the group associated with
the linear KdV equation, that is, for any $u_0$,  $U(t)u_0$ is the solution of the linear problem
\begin{equation}\label{lineargkdv}
\begin{cases}
\partial_t u+\partial_x^3u=0, \;\;x\in\R, \;t\in\R, \\
u(x,0)=u_0(x).
\end{cases}
\end{equation}

We begin by recalling the results necessary to prove Theorem \ref{local}.

\begin{lemma}\label{lemma1}
The following inequalities hold.
\begin{itemize}
  \item[(i)] If $u_0\in L^2_x$, then
  $$
   \|\partial_xU(t)u_0\|_{L^\infty_xL^2_T}\lesssim\|u_0\|_{L^2_x}.
   $$
  \item[(ii)] If $g\in L^1_xL^2_T$, then for any $T>0$,
  $$
  \left\|\partial_x\int_0^tU(t-t')g(\cdot,t')dt'\right\|_{L^\infty_TL^2_x}\lesssim\|g\|_{ L^1_xL^2_T}.
  $$
  \item[(iii)] If $g\in L^1_xL^2_T$, then for any $T>0$,
  $$
  \left\|\partial_x^2\int_0^tU(t-t')g(\cdot,t')dt'\right\|_{L^\infty_xL^2_T}\lesssim\|g\|_{ L^1_xL^2_T}.
  $$
\end{itemize}
\end{lemma}
\begin{proof}
See \cite[Lemma 3.5]{kpv1}.
\end{proof}

\begin{lemma}\label{lemma2}
The following inequalities hold.
\begin{itemize}
  \item[(i)] If $u_0\in L^2_x$, then
  $$
   \|U(t)u_0\|_{L^5_xL^{10}_T}\lesssim\|u_0\|_{L^2_x}.
   $$
  \item[(ii)] If $g\in L^{5/4}_xL^{10/9}_t$, then
  $$
  \left\|\int_0^tU(t-t')g(\cdot,t')dt'\right\|_{L^5_TL^{10}_x}\lesssim\|g\|_{ L^{5/4}_xL^{10/9}_T}.
  $$
\end{itemize}
\end{lemma}
\begin{proof}
See \cite[Corollary 3.8]{kpv1}.
\end{proof}

\begin{lemma}\label{lemma3}
If $g\in\mathcal{S}(\R^2)$, then
$$
\|g\|_{L^{5k/4}_xL^{5k/2}_t}\lesssim \left\|D^{\alpha_k}_xD^{\beta_k}_tg\right\|_{L^{p_k}_xL^{q_k}_t},
$$
where $\alpha_k,\beta_k,p_k$, and $q_k$ are defined as in \eqref{a4}-\eqref{a5}.
\end{lemma}
\begin{proof}
See \cite[Lemma 3.15]{kpv1}.
\end{proof}

\begin{lemma}\label{lemma4}
Let $s_k, \alpha_k, \beta_k, p_k$, and $q_k$ be as in Theorem \ref{local}. Let $\alpha\geq0$, $\beta\geq0$,
and $u_0\in\mathcal{S}(\R)$. Then
$$
\left\|D^\alpha_xD^{\beta/3}D^{\alpha_k}_xD^{\beta_k}_tU(t)u_0\right\|_{L^{p_k}_xL^{q_k}_t}
\lesssim \left\|D^{\alpha+\beta}_xD^{s_k}_xu_0\right\|_{L^2_x}.
$$
\end{lemma}
\begin{proof}
See \cite[Lemmas 3.14 and 3.16]{kpv1}.
\end{proof}

\begin{lemma}\label{lemma5}
Let $s\geq s_k$. Let $p_k$ and $q_k$ be as in Theorem \ref{local}. The following estimate is fulfilled
\begin{equation*}
\begin{split}
\left\|D^s_x(u^k\partial_xu)\right\|_{L^{5/4}_xL^{10/9}_t}\lesssim &
\left\|D^{\alpha_k}_xD^{\beta_k}_tu\right\|_{L^{p_k}_xL^{q_k}_t}^k\|D^s_x\partial_xu\|_{L^\infty_xL^2_t}\\
&+ \left\|D^{\alpha_k}_xD^{\beta_k}_tu\right\|_{L^{p_k}_xL^{q_k}_t}^{k-1}\|u\|_{L^{5k/4}_xL^{5k/2}_t}
\|D^s_x\partial_xu\|_{L^\infty_xL^2_t}.
\end{split}
\end{equation*}
\end{lemma}
\begin{proof}
See proof of Proposition 6.1 in \cite{kpv1}.
\end{proof}

\begin{lemma}\label{lemma6}
Let $0<\alpha<1$ and $p,p_1,p_2,q,q_1,q_2 \in (1,\infty)$ with
$\frac{1}{p}=\frac{1}{p_1}+\frac{1}{p_2}$ and
$\frac{1}{q}=\frac{1}{q_1}+\frac{1}{q_2}$. Then,
\begin{itemize}
    \item[(i)] $$
\|D^\alpha_x(fg)-fD^\alpha_xg-gD^\alpha_xf\|_{L^p_xL^q_T}\lesssim
\|D^\alpha_xf\|_{L^{p_1}_xL^{q_1}_T}\|g\|_{L^{p_2}_xL^{q_2}_T}.
$$
The same still holds if $p=1$ and $q=2$.
    \item[(ii)] $$\|D^\alpha_xF(f)\|_{L^p_xL^q_T}\lesssim
    \|D^\alpha_xf\|_{L^{p_1}_xL^{q_1}_T}\|F'(f)\|_{L^{p_2}_xL^{q_2}_T}$$
  \end{itemize}
\end{lemma}
\begin{proof}
See  \cite[Theorems A.6, A.8, and A.13]{kpv1}.\\
\end{proof}

Next, we introduce some tools to prove a variant of Theorem \ref{local}. These tools will be used in Section \ref{impglobal}.

We shall take advantage of the Strichartz estimates (see  \cite{KPV4})
\begin{equation}\label{L8}
\|u\|_{L^8_{x,t}}\lesssim \|u\|_{X_{0,\frac{1}{2}+}}
\end{equation}
and
\begin{equation*}
\|u\|_{L^{5}_{x}L^{10}_{t}}\lesssim \|u\|_{X_{0,\frac{1}{2}+}}.
\end{equation*}
By duality
\begin{equation*}
\|u\|_{X_{0,-\frac{1}{2}-}}\lesssim \|u\|_{L^{5/4}_{x}L^{10/9}_{t}},
\end{equation*}
which interpolated with the trivial estimate
\begin{equation*}
\|u\|_{X_{0,0}}\lesssim \|u\|_{L^2_{x,t}}
\end{equation*}
yields
\begin{equation}\label{L6}
\|u\|_{X_{0,-\frac{1}{2}++}}\lesssim \|u\|_{L^{5/4+}_{x}L^{10/9+}_{t}},
\end{equation}
where $a++=a+2\varepsilon$ for sufficiently small $\varepsilon>0$.

Interpolation between \eqref{L8} and $\|u\|_{L^{\infty}_{x,t}}\lesssim \|u\|_{X_{\frac{1}{2}+,\frac{1}{2}+}}$ gives us
\begin{equation}\label{LP}
\|u\|_{L^p_{x,t}}\lesssim \|u\|_{X_{\alpha(p),\frac{1}{2}+}},
\end{equation}
where $p>8$ and $\alpha(p)=\left(\dfrac{1}{2}+\right)\left(\dfrac{p-8}{p}\right)$.

Recall that from Lemmas \ref{lemma3} and \ref{lemma4},
\begin{equation*}
\left\|D^{\alpha_k}_xD^{\beta_k}_tu\right\|_{L^{p_k}_xL^{q_k}_t}\lesssim \left\|u\right\|_{X_{s_k,\frac{1}{2}+}}
\end{equation*}
and
\begin{equation*}\label{LI}
\|u\|_{L^{5k/4}_xL^{5k/2}_t}\lesssim \left\|D^{\alpha_k}_xD^{\beta_k}_tu\right\|_{L^{p_k}_xL^{q_k}_t}.
\end{equation*}
Interpolating, respectively, with $\|u\|_{L^{\infty}_{x,t}}\lesssim \|u\|_{X_{\frac{1}{2}+,\frac{1}{2}+}}$ and
$\|u\|_{L^{\infty}_{x,t}}\lesssim \|u\|_{L^{\infty}_{x,t}}$ we obtain
\begin{equation}\label{LI1}
\left\|D^{\alpha_k-}_xD^{\beta_k-}_tu\right\|_{L^{p_k+}_xL^{q_k+}_t}\lesssim \left\|u\right\|_{X_{s_k+,\frac{1}{2}+}}
\end{equation}
and
\begin{equation}\label{LI2}
\|u\|_{L^{5k/4+}_xL^{5k/2+}_t}\lesssim \left\|D^{\alpha_k-}_xD^{\beta_k-}_tu\right\|_{L^{p_k+}_xL^{q_k+}_t}.
\end{equation}
Moreover, by interpolation it follows that
\begin{equation}\label{LI11}
\|D^{s}_x u\|_{L^{p_3}_xL^{q_3}_t}\le c\|u\|_{L^{5k/4}_xL^{5k/2}_T}^{1-\theta_1}\|D^s_x\partial_xu\|_{L^{\infty}_xL^{2}_T}^{\theta_1}
\end{equation}
and
\begin{equation}\label{LI22}
\|\partial_x u\|_{L^{p_2}_xL^{q_2}_t}\le c\|u\|_{L^{5k/4}_xL^{5k/2}_T}^{1-\theta_2}\|D^s_x\partial_xu\|_{L^{\infty}_xL^{2}_T}^{\theta_2},
\end{equation}
where
\begin{equation}\label{P2P3}
\frac{1}{p_2}+\frac{1}{p_3}=\frac{4}{5k} \;\;\;\text{and}\;\;\; \frac{1}{q_2}+\frac{1}{q_3}=\frac12+\frac{2}{5k}
\end{equation}
and $\theta_1=\dfrac{s}{1+s}$ and $\theta_2=\dfrac{1}{1+s}$,  both $\theta_1$ and $\theta_2$ are in
(0,1) and $\theta_1+\theta_2=1$.

Finally, we have the following refined Strichartz estimate in the case of differing frequencies
(see Bourgain \cite{B1} and Gr\"unrock \cite{Gr05}).

\begin{lemma}\label{L3}
Let $\psi_1, \psi_2 \in X_{0,\frac{1}{2}+}$ be supported on spatial frequencies $|\xi_i|\sim N_i$, $i=1,2$.
If $\max\{|\xi_1|,|\xi_2|\} \lesssim \min\left\{|\xi_1-\xi_2|,|\xi_1+\xi_2|\right\}$ for all
$\xi_i\in \textrm{supp}(\widehat{\psi}_i)$, $i=1,2$, then
\begin{equation}\label{GRU}
\|\psi_1D_x\psi_2\|_{L^2_{x,t}}\lesssim \|\psi_1\|_{X_{0,\frac{1}{2}+}}
 \|\psi_2\|_{X_{0,\frac{1}{2}+}}.
\end{equation}
\end{lemma}
\begin{proof} See \cite[Lemma 2.1]{LG6}.
\end{proof}

We now give some useful notation for multilinear expressions. If $n \geq 2$ is an even integer, we define a
(spatial) $n$-multiplier to be any function $M_n(\xi_1,\dots,\xi_n)$ on the hyperplane
$$\Gamma_n\equiv\{ (\xi_1,\dots,\xi_n)\in \R^n : \xi_1+\cdots+\xi_n=0\},$$
which we endow with the standard measure $\delta(\xi_1+\cdots+\xi_n)$, where $\delta$ is the Dirac delta.

If $M_n$ is an $n$-multiplier and $f_1, \dots , f_n$ are functions on $\R$, we define the $n$-linear functional
$\Lambda_n(M_n; f_1, \dots , f_n)$ by
$$\Lambda_n(M_n; f_1, \dots , f_n)= \int_{\Gamma_n}M_n(\xi_1,\dots,\xi_n)\prod_{j=1}^n\widehat{f_j}(\xi_j).$$

We will often apply $\Lambda_n$ to $n$ copies of the same function $u$ in which case the dependence upon $u$ may
be suppressed in the notation: $\Lambda_n(M_n; u, \dots , u)$ may simply be written as $\Lambda_n(M_n)$.

If $M_n$ is symmetric, so does the  $n$-linear functional $\Lambda_n(M_n)$.

As an example, suppose that $u$ is an $\R$-valued function. By Plancherel, we can rewrite the energy (\ref{EC})
in terms of $n$-linear functionals as
\begin{eqnarray*}
E(u(t))=-\dfrac{1}{2}\Lambda_2(\xi_1\xi_2)-\dfrac{\mu}{k+2}\Lambda_{k+2}(1).
\end{eqnarray*}

The time derivative of a symmetric $n$-linear functional can be calculated explicitly if we assume that the function
$u$ satisfies a particular PDE. The following statement may be directly verified by using the generalized KdV
equation (\ref{gkdv}).

\begin{proposition}
Suppose u satisfies the  generalized KdV equation (\ref{gkdv}) and that $M_n$ is a symmetric $J$-multiplier. Then
\begin{equation}\label{LD}
\dfrac{d}{dt}\Lambda_J(M_J)\!\!=\!\!\Lambda_n(M_J\alpha_J)-iJ\mu \Lambda_{J+k}(M_n(\xi_1,\dots,\xi_{J-1},\xi_J+\cdots+\xi_{J+k})
(\xi_J+\cdots+\xi_{J+k})),
\end{equation}
where $\alpha_n\equiv i(\xi_1^3+\cdots+\xi_J^3)$.
\end{proposition}


\section{Local well-posedness}\label{localwp}

Our aim in this section is to establish Theorem \ref{local}. We use the contraction
mapping principle. Define the metric spaces
$$
X_T=\{u\in C([0,T];H^s(\mathbb{R})):\;\tres u\tres_s<\infty\}
$$
and
$$
X_T^a=\{u\in X_T:\;\tres u\tres_s\leq a\},
$$
where
\begin{equation}\label{norms}
\begin{split}
\tres
u\tres_s=&\|u\|_{L^\infty_TH^s_x}+\|u\|_{L^5_xL^{10}_T}+\|D^s_xu\|_{L^5_xL^{10}_T}\\
&+\|\partial_xu\|_{L^\infty_xL^{2}_T}+\|D^{s}_x\partial_xu\|_{L^\infty_xL^{2}_T}+\|D^{\gamma_k}_tD^{\alpha_k}_xD^{\beta_k}_tu\|_{L^{p_k}_xL^{q_k}_T}.
\end{split}
\end{equation}
The parameters $T$ and $a$ will be appropriately chosen later. On
$X_T$ consider the integral operator
\begin{equation}\label{Phi}
\Phi(u)(t):=U(t)u_0-\mu\int_0^tU(t-t')\partial_x(u^{k+1})(t')dt'.
\end{equation}

We only give the details to estimate the $\|\cdot\|_{L^{\infty}_TH^s}$--norm.
From group properties and Lemma \ref{lemma1} (ii),
\begin{equation*}
\begin{split}
\|\Phi(u)\|_{L^2_x}&\leq \|u_0\|_{L^2_x}+\|\partial_x\int_0^tU(t-t')u^{k+1}(t')dt'\|_{L^2_x} \\
&\lesssim|u_0\|_{L^2_x}+\|u^{k+1}\|_{L^1_xL^2_T}\\
&\lesssim\|u_0\|_{L^2_x}+\|u\|_{L^{5k/4}_xL^{5k/2}_T}^k\|u\|_{L^5_xL^{10}_T}.
\end{split}
\end{equation*}
Now, from Lemma \ref{lemma3} and Sobolev's inequality it follows that
\begin{equation}\label{a5.1}
\begin{split}
\|\Phi(u)\|_{L^2_x}&\leq \|u_0\|_{L^2_x}+\|D^{\alpha_k}_xD^{\beta_k}_tu\|_{L^{p_k}_xL^{q_k}_T}^k\|u\|_{L^5_xL^{10}_T}\\
&\lesssim \|u_0\|_{L^2_x}+T^{k\gamma_k}\|D^{\gamma_k}_tD^{\alpha_k}_xD^{\beta_k}_tu\|_{L^{p_k}_xL^{q_k}_T}^k\|u\|_{L^5_xL^{10}_T}\\
&\leq \|u_0\|_{L^2_x}+T^{k\gamma_k}\tres u\tres_s^{k+1}.
\end{split}
\end{equation}

Next, we estimate the $\dot{H}^s$-norm. Group properties and an application of Lemma \ref{lemma1} yield
\begin{equation*}
\begin{split}
\|D^s_x\Phi(u)\|_{L^2_x}&\leq \|D^s_xu_0\|_{L^2_x}+\|\partial_x\int_0^tU(t-t')D^s_x(u^{k+1})dt'\|_{L^2_x} \\
&\lesssim\|u_0\|_{L^2_x}+\|D^s_x(u^{k+1})\|_{L^1_xL^2_T}.
\end{split}
\end{equation*}
By applying Lemma \ref{lemma6} and then Lemma \ref{lemma3}, we deduce
\begin{equation}\label{a6a}
\begin{split}
\|D^s_x(u^{k+1})&\|_{L^1_xL^2_T}\lesssim
\|u^k\|_{L^{5/4}_xL^{5/2}_T}\|D^s_xu\|_{L^5_xL^{10}_T}
+\|uD^s_x(u^k)\|_{L^1_xL^2_T}\\
&\lesssim\|u\|_{L^{5k/4}_xL^{5k/2}_T}^k\|D^s_x u\|_{L^5_xL^{10}_T}
+\|u\|_{L^{5k/4}_xL^{5k/2}_T}\|D^s_x(u^k)\|_{L^{p_0}_xL^{q_0}_T}\\
&\lesssim\|D^{\alpha_k}_xD^{\beta_k}_tu\|_{L^{p_k}_xL^{q_k}_T}^k\|D^s_xu\|_{L^5_xL^{10}_T}\\
&\quad+\|D^{\alpha_k}_xD^{\beta_k}_tu\|_{L^{p_k}_xL^{q_k}_T}\|D^s_xu\|_{L^5_xL^{10}_T}\|u^{k-1}\|_{L^{p_1}_xL^{q_1}_T},
\end{split}
\end{equation}
where
$$
\frac{1}{p_1}\!=\!\frac{1}{p_0}-\frac{1}{5}\!=\!1-\frac{4}{5k}-\frac{1}{5}\!=\!\frac{4(k-1)}{5k}
\;\;\;
\text{and}
\;\;\;
\frac{1}{q_1}\!=\!\frac{1}{q_0}-\frac{1}{10}\!=\!\frac{1}{2}-\frac{2}{5k}-\frac{1}{10}\!=\!\frac{4(k-1)}{10k}.
$$

On the other hand, from Lemma \ref{lemma3},
\begin{equation}\label{a6}
\|u^{k-1}\|_{L^{p_1}_xL^{q_1}_T}\lesssim\|u\|_{L^{5k/4}_xL^{5k/2}_T}^{k-1}
\lesssim\|D^{\alpha_k}_xD^{\beta_k}_tu\|_{L^{p_k}_xL^{q_k}_T}^{k-1}.
\end{equation}
Sobolev's inequality and \eqref{a6} then imply
\begin{equation*}
\begin{split}
\|D^s_x(u^{k+1})\|_{L^1_xL^2_T}
&\lesssim\|D^{\alpha_k}_xD^{\beta_k}_tu\|_{L^{p_k}_xL^{q_k}_T}^k\|D^s_xu\|_{L^5_xL^{10}_T}\\
&\lesssim T^{k\gamma_k}\|D^{\gamma_k}_tD^{\alpha_k}_xD^{\beta_k}_tu\|_{L^{p_k}_xL^{q_k}_T}^k
\|D^s_xu\|_{L^5_xL^{10}_T}.
\end{split}
\end{equation*}
Therefore,
\begin{equation}\label{a7}
\|D^s_x\Phi(u)\|_{L^2_x}\leq \|D^s_xu_0\|_{L^2_x}+T^{k\gamma_k}\tres
u\tres_s^{k+1}.
\end{equation}

To estimate the remainder norms in \eqref{norms} we will make use of Lemmas  \ref{lemma1}, \ref{lemma2},
\ref{lemma5} and \ref{lemma3} to lead to

$$
\tres \Phi(u)\tres_s \leq c\|u_0\|_{H^s}+cT^{k\gamma_k}\tres
u\tres_s^{k+1}.
$$
Choose $a=2c\|u_0\|_{H^s}$ and $T>0$ such that
$$
ca^kT^{k\gamma_k}<\frac{1}{20}.
$$
This implies that $\Phi:X_T^a\rightarrow X_T^a$ is well defined. To
finish the proof we need to prove that $\Phi$ is also a contraction;
but, the argument is analogue to the previous one. The rest of the
proof follows in a standard way.

\begin{remark}
From the proof of Theorem \ref{local} it follows that
$$
T\sim\|u_0\|_{H^s}^{-1/\gamma_k}=\|u_0\|_{H^s}^{-3/(s-s_k)}.
$$
This is in agreement with the case $k=4$, where $T\sim \|u_0\|_{H^s}^{-3/s}$ (see \cite{FLP1}).
\end{remark}


\section{Global well-posedness in $H^1$}

In this section, we intend to show Theorem \ref{global3}. We begin by recalling the classical result obtained by Nagy \cite{Na}
(see also Weinstein \cite{W83}), regarding the best constant of the Gagliardo-Nirenberg inequality \eqref{ap2}.

\begin{theorem}\label{best}
Let $k>0$,  then the Gagliardo-Nirenberg inequality
\begin{equation}\label{g-n}
\|u\|_{L^{k+2}(\R)}^{k+2}\le K_{\rm opt}^{k+2}\,\|\nabla u\|_{L^2(\R)}^{\frac k 2}\|u\|_{L^2(\R)}^{2+\frac k 2},
\end{equation}
holds, and the sharp constant $K_{\rm opt}>0$ is
explicitly given by
\begin{equation}\label{opt1}
K_{\rm opt}^{k+2}=\frac{k+2}{2\|\psi\|_{L^2}^k},
\end{equation}
where $\psi$ is the unique non-negative, radially-symmetric, decreasing solution of the equation
\begin{equation}\label{ground}
\frac{k}{4}\Delta \psi-\left(1+\frac k 4 \right)\psi+\psi^{k+1}=0.
\end{equation}
\end{theorem}
\begin{proof}
See \cite{Na} and \cite{W83}.
\end{proof}

Before proceeding to our main result, we will establish a relation between the solution $\psi$ of  (\ref{ground})
and the unique non-negative, radially-symmetric, decreasing solution,  $Q$, of the equation
\begin{equation}\label{ground3}
\Delta Q-Q+Q^{k+1}=0.
\end{equation}

\begin{remark}\label{cgkdv} Recall that for the critical generalized KdV equation, that is, equation (\ref{gkdv})
with $k=4$, $\mu=1$, we have global solutions if $\|u_0\|_{L^2}<\|Q\|_{L^2}$, and $u_0\in H^s(\R)$, $s>6/13$, where
$Q$ is the solution of (\ref{ground3}) with $k=4$ (see  \cite{miao}, \cite{LG6}, \cite{FLP1}, and \cite{W83}).
\end{remark}

First, we note that if $\psi$ is a solution of (\ref{ground}) then $\lambda \psi(\omega x)$, where
$\lambda=\left(\frac{4}{k+4}\right)^{1/k}$ and $\omega=\left(\frac{k}{k+4}\right)^{1/2}$, is a solution of
(\ref{ground3}). Therefore, by uniqueness, we have
$$Q(x)=\lambda \psi(\omega x).$$
A simple calculation shows that
$$\|Q\|_{L^2}^2=\dfrac{\lambda^2}{\omega}\|\psi\|_{L^2}^2.$$
Combining this last relation with (\ref{opt1}) yields
\begin{eqnarray}\label{Q}
K_{\rm opt}^{k+2}=\dfrac{2(k+2)(k+4)^{\frac{k-4}{4}}}{(k)^{\frac{k}{4}}\| Q\|_{L^{2}}^{k}}.
\end{eqnarray}
Moreover, by multiplying (\ref{ground3}) by $Q$, integrating, and applying integration by parts, we obtain
$$\|Q\|_{L^{k+2}}^{k+2}=\|Q\|_{L^{2}}^{2}+\|\partial_x Q\|_{L^{2}}^{2}.$$

On the other hand, by multiplying (\ref{ground3}) by $x \partial_x Q$, integrating, and applying integration by parts,
we obtain the Pohozhaev-type identity
$$\dfrac{2}{k+2}\|Q\|_{L^{k+2}}^{k+2}=\|Q\|_{L^{2}}^{2}-\|\partial_x Q\|_{L^{2}}^{2}.$$
Combining the last two relations, we obtain
\begin{equation}\label{relQ}
\dfrac{(k+4)}{2(k+2)}\|Q\|_{L^{k+2}}^{k+2}=\|Q\|_{L^{2}}^{2} \textrm{  and  } \|Q\|_{L^{2}}^{2}=
\dfrac{k+4}{k}\|\partial_x Q\|_{L^{2}}^{2}.
\end{equation}

Now we are ready to prove the main global result of this section.

\begin{proof}[Proof of Theorem \ref{global3}] We proceed as follows: write the $\dot{H}^1$-norm of $u(t)$
using the quantities $M(u(t))$ and $E(u(t))$. Then we use the sharp Gagliardo-Nirenberg inequality \eqref{g-n} to yield
\begin{equation}\label{ap10}
\begin{split}
\|\partial_x u(t)\|_{L^2}^2&=2E(u_0)+\frac{2}{k+2} \|u(t)\|_{L^{k+2}}^{k+2}\\
&\le 2E(u_0)+\frac{2}{k+2}K_{\rm opt}^{k+2}\,
\|u_0\|_{L^2}^{\frac{k+4}{2}}\|\partial_x u(t)\|_{L^2}^{\frac{k}{2}}.
\end{split}
\end{equation}
Let $X(t)=\|\partial_x u(t)\|_{L^2}^2$, $A=2E(u_0)$,
and $B=\frac{2}{k+2}K_{\rm opt}^{k+2}\|u_0\|_{L^2}^{\frac{k+4}{2}}$, then we can write \eqref{ap10} as
\begin{equation}\label{ap12}
X(t)-B\,X(t)^{k/4}\le A, \text{\hskip2pt for}\;\;t\in (0,T),
\end{equation}
where $T$ is given by Theorem \ref{local}.

Now let $f(x)=x-B\,x^{k/4}$, for $x\ge 0$. The function $f$ has a local maximum at $x_0=\Big(\dfrac{4}{kB}\Big)^{4/(k-4)}$
with maximum value $f(x_0)=\dfrac{k-4}{k}\Big(\dfrac{4}{kB}\Big)^{4/(k-4)}.$
If we require that
\begin{equation}\label{ap13}
2E(u_0) < f(x_0)\,\,\,\, \mbox{and}\, \,\,\, X(0) < x_0,
\end{equation}
the continuity of $X(t)$ implies that $X(t) < x_0$ for any $t$ as long as the solution exists.

Using relations (\ref{relQ}), we have
$$E(Q)=\dfrac{k-4}{2(k+4)}\|Q\|_{L^2}^2.$$
Therefore, a simple calculation shows that conditions (\ref{ap13}) are exactly the inequalities (\ref{GR1}) and (\ref{GR2}).
Moreover the inequality $X(t) < x_0$ reduces to (\ref{GR3}). The proof of Theorem \ref{global3} is thus completed.
\end{proof}



\section{Global well-posedness in $H^s$, $s<1$: $\mu=-1$ and $k$ even}\label{impglobal}

In this section, we prove Theorem \ref{T1}. As we mentioned in the introduction,
we follow the ``almost conservation law" scheme introduce in  \cite{CKSTT4}--\cite{CKSTT2}.

\subsection{Modified energy functional}
To start with, we introduced a substitute notion of  ``energy" that could be defined for less regular functions and that has
very low increment in time. Given $s<1$ and a parameter $N\gg 1$, define a multiplier operator $I_N:H^s\rightarrow H^1$
such that
\begin{equation*}
\widehat{I_Nf(\xi)}\equiv m_N(\xi)\widehat{f}(\xi),
\end{equation*}
where the multiplier $m_N(\xi)$ is the nondecreasing in $|\xi|$, smooth and radially symmetric function defined as
\begin{eqnarray*}
m_N(\xi)=\left\{
\begin{array}{ll }
1&, \textrm{ if } |\xi|\leq N,\\
\left(\dfrac{N}{|\xi|}\right)^{1-s}&, \textrm{ if } |\xi|\geq 2N.
\end{array} \right.
\end{eqnarray*}

To simplify the notation, we omit the dependence of $N$ in $I_N$ and denote it only by $I$. Note that the operator $I$ is
smoothing of order $1-s$. Indeed we have
\begin{equation}\label{smo}
\|u\|_{X_{s_0,b_0}}\leq c\|Iu\|_{X_{s_{0}+1-s,b_0}}\leq cN^{1-s}\|u\|_{X_{s_0,b_0}},
\end{equation}
for any $s_0,b_0\in \R$.

Our substitute energy will be defined by $E^1(u)=E(Iu)$. Obviously this energy makes sense even if $u$ is only in
$H^s(\R)$. Thus,  in terms of $n$-linear functionals we have
\begin{equation}\label{ME1}
E^1(u)=-\dfrac{1}{2}\Lambda_2(m_1\xi_1m_2\xi_2)-\dfrac{\mu}{k+2}\Lambda_{k+2}(m_1\dots m_{k+2}),
\end{equation}
where $m_j=m(\xi_j)$.

We can think about $E^1(u)$ as the first generation of a family of modified energies. One can also define the second energy
\begin{equation}\label{ME2}
E^2(u)=-\dfrac{1}{2}\Lambda_2(m_1\xi_1m_2\xi_2)-\dfrac{\mu}{k+2}\Lambda_{k+2}(M_{k+2}(\xi_1,\dots, \xi_{k+2})),
\end{equation}
where $M_{k+2}$ is an arbitrarily symmetric $(k+2)$-multiplier.

Thus, using the derivation law (\ref{LD}), we obtain
\begin{equation*}
\begin{split}
\frac{d}{dt}E^2(u)
&=-\dfrac{1}{2}\Lambda_2(m_1\xi_1m_2\xi_2\alpha_2)\\
&\quad+\dfrac{\mu i}{k+2}\Lambda_{k+2}((m_1^2\xi_1^3+\dots+m_{k+2}^2\xi_{k+2}^3)-M_{k+2}(\xi_1^3+\dots+\xi_{k+2}^3))\\
&\quad+\mu^2i\Lambda_{2k+2}(M_{k+2}(\xi_1,\dots,\xi_{k+1},\xi_{k+2}+\cdots+\xi_{2k+2})(\xi_{k+2}+\cdots+\xi_{2k+2})),
\end{split}
\end{equation*}
where we have used the identity $\xi_1+\cdots+\xi_6=0$ and symmetrizing.

Note that picking
\begin{equation*}
M_{k+2}(\xi_1,\dots,\xi_{k+2})=\dfrac{m_1^2\xi_1^3+\dots+m_{k+2}^2\xi_{k+2}^3}{\xi_1^3+\dots+\xi_{k+2}^3}
\end{equation*}
we can force $\Lambda_{k+2}$ to be zero. Unfortunately the multiplier $M_{k+2}$ is not well defined in the set $\Gamma_{k+2}$.
In fact, given $N\gg 1$, we can find numbers $\xi_1,\dots,\xi_{k+2}$ such that the denominator of $M_{k+2}$ is zero and the numerator
is different from zero. This is the content of the following proposition.

\begin{proposition}\label{P1}
There exist numbers $\xi_1,\dots,\xi_{k+2}$ such that
\begin{itemize}
\item[(i)]
$\left\{
\begin{array}{l }
\xi_1+\cdots+\xi_{k+2}= 0;\\
\xi_1^3+\cdots+\xi_{k+2}^3=0.
\end{array} \right.$
\item[(ii)] $m_1^2\xi_1^3+\cdots+m_{k+2}^2\xi_{k+2}^3\neq 0.$
\end{itemize}
\end{proposition}
\begin{proof}
See \cite[Proposition 3.1 and Remark 3.2]{LG6} .
\end{proof}

Therefore, throughout this section we work only with the first modified energy (\ref{ME1}).
Again, using the derivation law (\ref{LD}) and symmetrizing, we have
\begin{equation*}
\begin{split}
\frac{d}{dt}E^1(u)&=-\dfrac{1}{2}\Lambda_2(m_1\xi_1m_2\xi_2(\xi_1^3+\xi_2^3))\\
&+\dfrac{\mu i}{k+2}\Lambda_{k+2}((m_1^2\xi_1^3+\dots+m_{k+2}^2\xi_{k+2}^3)-m_1\dots m_{k+2}(\xi_1^3+\dots+\xi_{k+2}^3))\\
&+\mu^2i\Lambda_{2k+2}(m_1\dots m_{k+1}m(\xi_{k+2}+\cdots+\xi_{2k+2})(\xi_{k+2}+\cdots+\xi_{2k+2})).
\end{split}
\end{equation*}

\begin{remark}Observe that if $m=1$, the $\Lambda_{k+2}$ term vanish trivially. On the other hand, the terms
$\Lambda_2$ and $\Lambda_{2k+2}$ are also zero, since we have the restriction $\xi_1+\xi_2=0$ in the first and
symmetrization in the later. This reproduces the Fourier proof of the energy conservation (\ref{EC}).
\end{remark}

As one particular instance of the above computations and the Fundamental Theorem of Calculus, we have\\
\begin{eqnarray}\label{EMC1}
E^1(u)(t)-E^1(u)(0)=\int_0^t\frac{d}{dt}E^1(u)(t')dt'=
\end{eqnarray}
\vspace*{-0.25in}
\begin{equation*}
\begin{split}
=&\dfrac{\mu i}{k+2}\int_0^t\Lambda_{k+2}((m_1^2\xi_1^3+\dots+m_{k+2}^2\xi_{k+2}^3)-m_1\dots m_{k+2}(\xi_1^3
+\!\dots\!+\xi_{k+2}^3)) (t')dt'\\
&+\mu^2i\int_0^t\Lambda_{2k+2}(m_1\dots m_{k+1}m(\xi_{k+2}+\cdots+\xi_{2k+2})(\xi_{k+2}+\cdots+\xi_{2k+2}))(t')dt'.
\end{split}
\end{equation*}

Most of our arguments here consist in  showing that the quantity $E^1(u)$ is \textit{almost} conserved in time.

\subsection{Almost conservation law}
This subsection presents a detailed analysis of the expression (\ref{EMC1}). The analysis identifies some
cancelations in the pointwise upper bound of some multipliers depending on the relative size of the frequencies
involved. Our aim is to prove the following almost conservation property.

\begin{proposition}\label{p4.1}
Let $s>1/2$, $N\gg 1$ and $u\in H^s(\R)$ be a solution of \eqref{gkdv} on $[T, T+\delta]$ such that $Iu\in H^1(\R)$.
Then the following estimate holds
\begin{equation}\label{CC}
\left|E^1(u)(T+\delta)-E^1(u)(T)\right|\lesssim
N^{-2+}\left( \left\|Iu\right\|_{X^{\delta}_{1,\frac{1}{2}+}}^{k+2}
+\left\|Iu\right\|_{X^{\delta}_{1,\frac{1}{2}+}}^{2k+2}\right).
\end{equation}
\end{proposition}

\begin{remark}
The exponent $-2+$ on the right hand side of (\ref{CC}) is directly tied to the restriction $s>\dfrac{3+2(1/2-2/k)}{5}$
in our main theorem. If one could replace the increment $N^{-2+}$ by $N^{-\alpha+}$ for some $\alpha>0$ the argument we
give in Section \ref{Global} implies global well-posedness of (\ref{gkdv}) for all $s>\dfrac{3+\alpha(1/2-2/k)}{3+\alpha}$.
\end{remark}
\begin{proof} We start with the estimate for the $\Lambda_{k+2}$ term. Instead of estimating each multilinear expression
separately, we shall exploit some cancelation between the two multipliers. Using symmetrization and the fact that
$\xi_1+\cdots+\xi_{k+2}=0$ this term can be rewritten as
\begin{equation*}
\Lambda_{k+2}((m_1^2\xi_1^3+\dots+m_{k+2}^2\xi_{k+2}^3)-m_1\dots m_{k+2}(\xi_1^3+\dots+\xi_{k+2}^3))
\end{equation*}
\begin{equation*}
=(k+2)\int_{\ast} \left( \frac{m(\xi_2+\cdots+\xi_{k+2})}{m(\xi_2)\cdots m(\xi_{k+2})}-1\right) \xi_1^3\widehat{{Iu(\xi_1)}}
\cdots
 \widehat{Iu(\xi_{k+2})},
\end{equation*}
where $\ast$ denotes integration over $\xi_1+\cdots+\xi_{k+2}=0$.

Therefore, our aim is to obtain the following inequality
\begin{equation*}
\mathbf{Term}\lesssim N^{-2+}\prod_{i=1}^{k+2}\left\|I\phi_i\right\|_{X^{\delta}_{1,\frac{1}{2}+}},
\end{equation*}
where
\begin{equation*}
\mathbf{Term}\equiv \left|\int_0^{\delta}\!\!\!\int_{\ast}  \left( \frac{m(\xi_2+\cdots+\xi_6)}{m(\xi_2)\cdots m(\xi_6)}-1\right)
\xi_1^3 \widehat{{I\phi_1(\xi_1)}}
 \cdots
 \widehat{I\phi_{k+2}(\xi_{k+2})}\right|.
\end{equation*}

We estimate $\mathbf{Term}$ as follows. Without loss of generality, we assume that the Fourier transforms of
all these functions are non-negative. First, we bound the symbol in the parentheses pointwise
in absolute value, according to the relative sizes of the frequencies involved. After that, the
remaining integrals are estimated using Plancherel formula, H\"older's inequality and Lemma \ref{L3}.
To sum over the dyadic pieces at the end we need to have extra factors
$N_j^{0-}$, $j=1,\dots,k+2$, everywhere.

We decompose the frequencies $\xi_j$, $j=1,\dots,k+2$, into dyadic blocks $N_j$. By the symmetry
of the multiplier
\begin{equation}\label{MULT}
\frac{m(\xi_2+\cdots+\xi_{k+2})}{m(\xi_2)\cdots m(\xi_{k+2})}-1
\end{equation}
in $\xi_2$, \dots, $\xi_{k+2}$, we may assume that
\begin{equation*}
N_2\geq \cdots \geq N_{k+2}.
\end{equation*}

Moreover, we can assume
$N_2 \gtrsim N$, because otherwise the symbol is zero. The condition $\sum_{i=1}^{k+2}\xi_i=0$
implies $N_1\lesssim N_2$. We split the different frequency interaction into several cases,
according to the size of the parameter $N$ in comparison to the $N_i$'s.\\

\quebra \textbf{Case $A$: }$N_2\gtrsim N\gg N_3\geq \cdots \geq N_{k+2}$.\\

The condition $\sum_{i=1}^{k+2}\xi_i=0$ implies $N_1\sim N_2$. By the mean value theorem,
\begin{equation*}
\left|\frac{m(\xi_2)-m(\xi_2+\cdots+\xi_{k+2})}{m(\xi_2)}\right|\lesssim
\frac{\left|\nabla m(\xi_2)(\xi_3+\cdots+\xi_{k+2})\right|}{m(\xi_2)}\lesssim\frac{N_3}{N_2}.
\end{equation*}

Therefore, Lemma \ref{L3} and the Sobolev embedding  imply that
\begin{eqnarray*}
{\mathbf{Term}} \!\!\!&\lesssim& \!\!\! \frac{N_1^3N_3}{N_2}\left\|I\phi_1I\phi_3\right\|_{L^2(\R\times[0,\delta])}
\left\|I\phi_2I\phi_4\right\|_{L^2(\R\times[0,\delta])} \prod_{i=5}^{k+2} \left\|I\phi_i\right\|_{L^{\infty}}\\
&\lesssim&  \!\!\!\frac{N_1^3N_3}{N_2N_1N_2N_1N_2\langle N_3 \rangle\langle N_4 \rangle
\prod_{i=5}^{k+2} \langle N_i \rangle^{1/2-}} \prod_{i=1}^{k+2} \|I\phi_i\|_{X^{\delta}_{1,\frac{1}{2}+}}\\
&\lesssim& \!\!\! N^{-2+}N_{max}^{0-}\prod_{i=1}^{k+2} \|I\phi_i\|_{X^{\delta}_{1,\frac{1}{2}+}}.
\end{eqnarray*}

The remaining cases $N_2\gg N_3\gtrsim N$ and $N_3\geq \cdots \geq N_6$ (\textbf{Case $B$}) and
$N_2\sim N_3\gtrsim N$ and $N_3\geq \cdots \geq N_6$ (\textbf{Case $C$}) can be done using the same arguments as
in Farah \cite{LG6} (just put the remaining terms $I\phi_j$, $j=7,\cdots, k+2$ in $L^{\infty}_{x,t}$ and apply the
Sobolev embedding).

Now we turn to the estimate of the $\Lambda_{2k+2}$ term. Before we start let us fix some notation.
We write $N_1^{\ast}\geq N_2^{\ast}\geq N_3^{\ast}$ for the highest, second highest and third highest values of the
frequencies $N_1,\dots,N_{2k+2}$. It is clear that
\begin{equation}\label{B10}
|m_1\dots m_{k+1}m(\xi_{k+2}+\cdots+\xi_{2k+2})(\xi_{k+2}+\cdots+\xi_{2k+2})|\lesssim N_1^{\ast}.
\end{equation}

Again we perform a Littlewood-Paley decomposition of the ten functions $u$.\\

 \textbf{Case $A$: }$N_1^{\ast}\sim N_2^{\ast}\sim N_3^{\ast} \gtrsim N. $\\

 In view of \eqref{B10} and the fact that $m^3(N_1^{\ast})N_1^{\ast3-}\gtrsim N^{3-}$, we have
  \begin{eqnarray*}
 \left|\int_T^{T+\delta}\Lambda_{2k+2}(m_1\dots m_{k+1}m(\xi_{k+2}+\cdots+\xi_{2k+2})(\xi_{k+2}+\cdots+\xi_{2k+2}))(t')dt'\right|
\end{eqnarray*}
 \vspace*{-0.25in}
 \begin{eqnarray*}
 &\lesssim& \dfrac{N^{\ast0-}_1}{N^{2-}}\int\!\!\!\int |JIu|^3|u|^{2k-1}\\
 &\lesssim&\dfrac{N^{\ast0-}_1}{N^{2-}}\|JIu\|_{L^8}^3\|u\|^{2k-1}_{8(2k-1)/5}\\
  &\lesssim&\dfrac{N^{\ast0-}_1}{N^{2-}}\|Iu\|_{X^{\delta}_{1,\frac{1}{2}+}}^3  \|u\|^7_{X^{\delta}_{\alpha(8(2k-1)/5),\frac{1}{2}+}},
 \end{eqnarray*}
where we have applied H\"older inequality, \eqref{L8} and \eqref{LP}.

Note that $\alpha(8(2k-1)/5)=(k-3)/(2k-1)+$. Therefore the inequality \eqref{smo} implies
$$\|u\|_{X^{\delta}_{\alpha(8(2k-1)/5),\frac{1}{2}+}}\lesssim \|Iu\|_{X^{\delta}_{1,\frac{1}{2}+}},$$
for all $s>(k-3)/(2k-1)$ (note that $(k-3)/(2k-1)< 1/2$).

So, in this case
\begin{equation*}
\begin{split}
\Bigg|\int_T^{T+\delta}\Lambda_{2k+2}(m_1\dots &m_{k+1}m(\xi_{k+2}+\cdots+\xi_{2k+2})(\xi_{k+2}+\cdots+\xi_{2k+2}))(t')dt'\Bigg|\\
&\lesssim \dfrac{N^{\ast0-}_1}{N^{2-}}\|Iu\|_{X^{\delta}_{1,\frac{1}{2}+}}^{2k+2}.
\end{split}
 \end{equation*}

 \textbf{Case $B$: }$N_1^{\ast}\sim N_2^{\ast}\gg N_3^{\ast}$.\\

 Let $u_j\equiv u(N_j)$. Again, the inequality $m^2(N_1^{\ast})N_1^{\ast2-}\gtrsim N^{2-}$ and \eqref{B10} implies that
 \begin{eqnarray*}
 \left|\int_T^{T+\delta}\Lambda_{2k+2}(m_1\dots m_{k+1}m(\xi_{k+2}+\cdots+\xi_{2k+2})(\xi_{k+2}
 +\cdots+\xi_{2k+2}))(t')dt'\right|
\end{eqnarray*}
 \vspace*{-0.25in}
 \begin{eqnarray*}
&\lesssim& \dfrac{N^{\ast0-}_1}{N^{1-}}\|JIu_1u_3\|_{L^2}\|JIu_2\prod_{j=4}^{2k+2}u_j\|_{L^2}\\
&\lesssim& \dfrac{N^{\ast0-}_1}{N^{2-}}\|JIu_1\|_{L^2}\|u_3\|_{L^2}\|JIu_2\|_{L^8}\|u\|^{2k-1}_{L^{8(2k-1)/3}}\\
&\lesssim&\dfrac{N^{\ast0-}_1}{N^{2-}}\|Iu\|_{X^{\delta}_{1,\frac{1}{2}+}}^{2k+2},
 \end{eqnarray*}
where we have applied H\"older inequality, Lemma \ref{L3}, \eqref{L8} and \eqref{LP} with
$\alpha(8(2k-1)/3)=(k-2)/(2k-1)+<1/2$. This concludes the proof of Proposition \ref{p4.1}.
\end{proof}

\subsection{Proof of Theorem \ref{T1}}\label{Global}

Before proceeding to the proof of Theorem \ref{T1} we will first establish local well-posedness for the generalized
KdV equation \eqref{gkdv} in the Bourgain spaces $X_{s,b}$. As in Theorem \ref{local}, by the Duhamel's principle,
we need to find a solution for the following integral equation
$$u(t)=U(t)u_0+\int^t_0 U(t-s)\partial_x(u^{k+1})(s)ds.$$

The proof proceeds by the usual fixed point argument. By well-known linear estimates, we have for all $s>s_k$
\begin{eqnarray}
\left\| u\right\|_{X_{s,1/2+}^I} &=& \left\| U(t)u_0+\int^t_0 U(t-s)\partial_x(u^{k+1})(s)ds\right\|_{X_{s,1/2+}^I}\\
&\lesssim& \left\| u_0\right\|_{X_{s,1/2+}^I}+T^{\varepsilon}\left\| \partial_x(u^{k+1})\right\|_{X_{s,-1/2++}^I},
\end{eqnarray}
for sufficiently small $\varepsilon>0$.

Thus, the crucial nonlinear estimate for the local existence is given in the following lemma.
\begin{lemma}\label{NLEL}
For $s>s_k=(k-4)/2k$ we have
\begin{equation}\label{NLE}
\|\partial_x(u^{k+1})\|_{X_{s,-\frac12++}}\lesssim \|u\|^{k+1}_{X_{s,\frac12+}}.
\end{equation}
\end{lemma}
\begin{proof} By the fractional Leibniz rule in Lemma \ref{lemma6}, inequality \eqref{L6}, and H\"older inequality,
we obtain
\begin{equation*}
\begin{split}
\|\partial_x(u^{k+1})\|_{X_{s,-\frac12++}} &= \|J^s \partial_x(u^{k+1})\|_{X_{0,-\frac12++}}
\lesssim \|J^s \partial_x(u^{k+1})\|_{L_x^{5/4+}L_t^{10/9+}}\\
&\lesssim \|J^s(u^k)\|_{L_x^{p_1+}L_t^{q_1+}} \|\partial_xu\|_{L_x^{p_2}L_t^{q_2}}+
\|u^k\|_{L_x^{5/4+}L_t^{5/2+}} \|J^s\partial_xu\|_{L_x^{\infty}L_t^{2}}\\
&\lesssim \|u^{k-1}\|_{L_x^{5k/4(k-1)+}L_t^{5k/2(k-1)+}} \|J^su\|_{L_x^{p_3}L_t^{q_3}}
\|\partial_xu\|_{L_x^{p_2}L_t^{q_2}}\\
&+ \|u\|^{k-1}_{L_x^{5k/4}L_t^{5k/2}} \|u\|_{L_x^{5k/4+}L_t^{5k/2+}}\|J^s\partial_xu\|_{L_x^{\infty}L_t^{2}},
\end{split}
\end{equation*}
where $p_2$ and $p_3$ are defined as in \eqref{P2P3}.

Therefore, an application of inequalities \eqref{LI11} and \eqref{LI22} followed by inequalities \eqref{LI1} and \eqref{LI2}
yield the desired estimate \eqref{NLE}.
\end{proof}

\begin{remark}
As a consequence, one can recover all the well known range of existence for the local theory in terms of the $X_{s,b}$ spaces.
\end{remark}

Next, we consider the following modified equation
\begin{eqnarray}\label{MKdV}
\left\{
\begin{array}{l}
Iu_{t}+Iu_{xxx}+I(u^{k+1})_{x}=0, \peq  x\in \R,\, t>0,\\
Iu(x,0)=Iu_0(x).
\end{array} \right.
\end{eqnarray}

Clearly if $Iu\in H^1(\R)$ is a solution of \eqref{MKdV}, then $u\in H^s(\R)$ is a solution of \eqref{gkdv} in the same time interval.
Therefore, we need to prove that, in fact, the above modified equation has a global solution.

Applying the interpolation lemma (see \cite{CKSTT2}, Lemma 12.1) to \eqref{NLE}, we obtain
\begin{equation*}
\|\partial_xI(u^{k+1})\|_{X_{1,-1/2++}}\lesssim \|Iu\|^{k+1}_{X_{1,1/2+}}.
\end{equation*}
where the implicit constant is independent of $N$. Now, standard arguments invoking the contraction-mapping principle give the following
variant local well-posedness result.

\begin{theorem}\label{t3.2}
Assume $s_k<s< 1$. Let $u_0 \in H^s(\R)$ be
given. Then there exists a positive number $\delta$ such
that the IVP (\ref{MKdV}) has a unique local solution $Iu\in C([0,\delta]:H^1(\R))$ such that
\begin{equation}\label{CONT}
\|Iu\|_{X^{\delta}_{1,\frac{1}{2}+}}\lesssim \|Iu_0\|_{H^1}.
\end{equation}

Moreover, the existence time can be estimated by
\begin{equation*}
\delta \sim \dfrac{1}{\|Iu_0\|^{\sigma}_{H^1}},
\end{equation*}
where $\sigma>0$.
\end{theorem}

Now, we have all tools to prove our global result stated in Theorem \ref{T1}.\\

\begin{proof}[Proof of Theorem \ref{T1}]Let $u_0 \in H^s(\R)$ with $s_k< s <1$. Our goal is to construct a solution to
\eqref{MKdV} (and therefore to \eqref{gkdv}) on an arbitrary time interval $[0,T]$.
We rescale the solution by writing $u_{\lambda}(x,t)=\lambda^{-2/k}u(x/\lambda,t/\lambda^3)$.
We can easily check that $u(x,t)$ is a solution of \eqref{gkdv} on the time interval $[0,T]$ if and only if
$u_{\lambda}(x,t)$ is a solution to the same equation, with initial data $u_{0,\lambda}=\lambda^{-2/k}u_0(x/\lambda)$,
on the time interval $[0,\lambda^3T]$.

Since $k$ is even we have $\int u^{k+2}(x,t)dx>0$, for all $t>0$, therefore for $\mu=-1$
\begin{equation}\label{H1B}
\|\partial_xIu_{\lambda}(t)\|^2_{L^2}\lesssim E(Iu_{\lambda})(t).
\end{equation}

On the other hand
\begin{eqnarray*}
E(Iu_{0,\lambda}) &\lesssim& \|\partial_xIu_{0,\lambda}\|^2_{L^2}+\|Iu_{0,\lambda}\|^{k+2}_{L^{k+2}}\\
&\lesssim& \left(N^{2(1-s)}\lambda^{-2(s-1/2+2/k)}+\lambda^{-(k+2)(2/k-1/k+2)}\right)\left(1+\|u_0\|_{H^s}\right)^{k+2}.
\end{eqnarray*}
where in the last inequality we have used that
\begin{equation}\label{IU01}
\|\partial_xIu_{0,\lambda}\|_{L^2}\lesssim N^{1-s}\||\xi|^s\widehat{u_{0,\lambda}}\|_{L^2}
=N^{1-s}\lambda^{-(s-1/2+2/k)}\|u_0\|_{\dot{H}^s}.
\end{equation}
and, by Sobolev embedding,
\begin{equation}\label{IU02}
\|Iu_{0,\lambda}\|_{L^{k+2}} \lesssim \|D^{1/2-1/k+2}Iu_{0,\lambda}\|_{L^2}
\lesssim \lambda^{-(2/k-1/k+2)}\|u_0\|_{{H}^s},
\end{equation}
for all $s>1/2-1/k+2$.

Now, we apply our variant local existence Theorem \ref{t3.2} on $[0, \delta]$, where
$\delta\sim \|Iu_{0,\lambda}\|^{-\sigma}_{H^1}$,
 $\sigma>0$, to conclude that
\begin{equation}\label{IUV}
\|Iu_{\lambda}\|_{X^{\delta}_{1,\frac{1}{2}+}}\lesssim \|Iu_{0,\lambda}\|_{H^1}.
\end{equation}

The choice of the parameter $N=N(T)$ will be made later, but we select $\lambda$ now by requiring
\begin{eqnarray*}
N^{2(1-s)}\lambda^{-2(s-1/2+2/k)}\left(1+\|u_0\|_{H^s}\right)^{k+2} <1\Longrightarrow \lambda
\sim N^{\frac{1-s}{s-1/2+2/k}}.
\end{eqnarray*}

\begin{remark}
Note that $2/k-1/k+2>0$.
\end{remark}

From now on, we drop the $\lambda$ subscript on $u$. By the almost conservation law stated in Proposition \ref{p4.1} and
\eqref{IU01}-\eqref{IUV}, we have
\begin{equation*}
E^1(1)\leq E^1(0)+cN^{-2+}<1+cN^{-2+}<2.
\end{equation*}
We iterate this process $M$ times obtaining
\begin{equation}\label{UB}
E^1(M)\leq E^1(0)+cMN^{-2+}<1+cMN^{-2+}<2,
\end{equation}
as long as $MN^{-2+}\lesssim 1$, which implies that the lifetime of the local results remains uniformly of size $1$.
We take $M\sim N^{2-}$.
This process extends the local solution to the time interval $[0,N^{2-}]$. Now, we choose $N=N(T)$ so that
$$N^{2-}>\lambda^3T\sim N^{3\left(\frac{1-s}{s-1/2+2/k}\right)}T\Longrightarrow N^{2-3\frac{1-s}{s-1/2+2/k}-}>T.$$
Therefore, if $s>\dfrac{4(k-1)}{5k}$ then $T$ can be taken arbitrarily large which conclude our global result.

Finally, we need to establish the polynomial bound (\ref{pb}). Undoing the scaling, we have
\begin{equation*}
\|\partial_xIu_{\lambda}(\lambda^3T_0)\|^2_{L^2}=\frac{1}{\lambda^{1+4/k}}\|\partial_xIu(T_0)\|^2_{L^2}.
\end{equation*}

Let $T_0\sim N^{2-3\frac{1-s}{s-1/2+2/k}-}$, therefore our uniform bound \eqref{UB} together with \eqref{smo},
\eqref{MC} and \eqref{H1B} imply
\begin{eqnarray*}
\|u(T_0)\|^2_{H^{s}}\lesssim \|Iu(T_0)\|^2_{H^{1}}&\lesssim& \|Iu(T_0)\|^2_{L^{2}}+ \|\partial_xIu(T_0)\|^2_{L^{2}}\\
&\lesssim& \|u_0\|^2_{L^{2}}+ \lambda^{1+4/k}\|\partial_xIu_{\lambda}(\lambda^3T_0)\|^2_{L^{2}}\\
&\lesssim& \|u_0\|^2_{L^{2}}+ N^{(1+4/k)\left(\frac{1-s}{s-1/2+2/k}\right)}\\
&\lesssim& (1+T_0)^{\frac{(1+4/k)(1-s)}{5s-4(k-1)/k}+}(1+\|u_0\|_{H^{s}})^2.
\end{eqnarray*}
The proof of Theorem \ref{T1} is thus completed.
\end{proof}

\begin{remark}\label{KODD}
It is not clear how to apply the $I$-method when $\mu=1$ or $\mu=-1$ and $k$ odd. In this case, we may not have
inequality \eqref{H1B}. Therefore, to perform the interactions explained above we need to verify the hypotheses
of Theorem \ref{global3} for the modified solution $Iu(t)$ at each step. However, the only available estimate in
the homogeneous $H^1$-Sobolev space is the following
\begin{equation*}
\|\partial_xIu_{t}\|_{L^2}\lesssim N^{1-s}\|u(t)\|_{\dot{H}^s}.
\end{equation*}
Since, at the end of the argument we need to take $N$ large, we cannot satisfy the inequalities \eqref{GR1}-\eqref{GR3}
during all the interactions.
\end{remark}


\bibliographystyle{mrl}

\begin{thebibliography}{99}


\bibitem{BKPSV96}
B.~Birnir, C.~E. Kenig, G.~Ponce, N.~Svanstedt, and L.~Vega,
\newblock On the ill-posedness of the {IVP} for the generalized {K}orteweg-de {V}ries and nonlinear
{S}chr\"odinger equations,
\newblock {\em J. London Math. Soc.} 53, 551--559, 1996.

\bibitem{B}
J.~Bourgain,
\newblock Fourier transform restriction phenomena for certain lattice subsets and applications to
nonlinear evolution equations, {I} and {II},
\newblock {\em Geom. Funct. Anal.} 3, 107--156, 209--262, 1993.

\bibitem{B1}
J.~Bourgain,
\newblock Refinements of {S}trichartz' inequality and applications to
  {$2$}{D}-{NLS} with critical nonlinearity,
\newblock {\em Internat. Math. Res. Notices} 1998, 253--283, 1998.



\bibitem{CCT} M. Christ, J. Colliander, and T. Tao, Asymptotics, frequency modulation and low
regularity ill-posedness for canonical defocusing equations, {\em Amer. J. Math.} 125, 1235--1293, 2003.



\bibitem{CKSTT4}
J.~Colliander, M.~Keel, G.~Staffilani, H.~Takaoka, and T.~Tao,
\newblock Almost conservation laws and global rough solutions to a nonlinear {S}chr\"odinger equation,
\newblock {\em Math. Res. Lett.} 9, 659--682, 2002.


\bibitem{CKSTT3}
J.~Colliander, M.~Keel, G.~Staffilani, H.~Takaoka, and T.~Tao,
\newblock Sharp global well-posedness for {K}d{V} and modified {K}d{V} on  {$\mathbb{R}$} and
{$\mathbb{T}$},
\newblock {\em J. Amer. Math. Soc.} 16, 705--749, 2003.

\bibitem{CKSTT2}
J.~Colliander, M.~Keel, G.~Staffilani, H.~Takaoka, and T.~Tao,
\newblock Multilinear estimates for periodic {K}d{V} equations, and
  applications,
\newblock {\em J. Funct. Anal.} 211, 173--218, 2004.


\bibitem{LG6}
L.~G. Farah,
\newblock Global rough solutions to the critical generalized kdv equation,  {\em Journal of Differential
Equations} 249, 1968--1985, 2010.


\bibitem{FLP1} G. Fonseca, F. Linares, and G. Ponce,  Global existence for the critical generalized
KdV equation, {\em Proc. Amer. Math. Soc.} 131, 1847--1855,  2003.

\bibitem{GPS} A. Gr\"unrock, M. Panthee, and J. D. Silva, A remark on global well-posedness below $L^2$
for the GKDV-3 equation, {\em Differential Integral Equations} 20, 1229--1236,  2007.

\bibitem{Gr05}
A.~Gr{\"u}nrock,
\newblock A bilinear {A}iry-estimate with application to g{K}d{V}-3,
\newblock {\em Differential Integral Equations} 18, 1333--1339, 2005.

\bibitem{Gu} Z. Guo,  Global well-posedness of Korteweg–de Vries equation in $H^{-3/4}(\mathbb{R})$,
{\em J. Math. Pures Appl.} 91, 583--597,  2009.

\bibitem{HR08} J. Holmer and S. Roudenko,  A sharp condition for scattering of the radial 3D cubic
nonlinear Schr\"odinger equation, {\em Commun. Math. Phys.} 282, 435--467, 2008.

\bibitem{K83}
T.~Kato,
\newblock On the {C}auchy problem for the (generalized) {K}orteweg-de {V}ries   equation,
\newblock {\em Studies in applied mathematics}, volume~8 of {\em Adv. Math. Suppl. Stud.},
93--128, Academic Press, New York, 1983.

\bibitem{kenig-merle}
C.~E. Kenig and F. Merle,
\newblock Global well-posedness, scattering, and blow-up for the energy-critical
focusing nonlinear Schr\"odinger equation in the radial case,
\newblock {\em Invent. Math.} 166, 645–-675, 2006.

\bibitem{KPV4}
C.~E. Kenig, G.~Ponce, and L.~Vega,
\newblock Oscillatory integrals and regularity of dispersive equations,
\newblock {\em Indiana Univ. Math. J.} 40, 33--69, 1991.

\bibitem{kpv1}
C.~E. Kenig, G.~Ponce and L.~Vega,
\newblock Well-posedness and scattering results for the generalized
Korteweg-de Vries equation via the contraction principle,
\newblock {\em Comm. Pure Appl. Math.} 46, 527--620, 1993.

\bibitem{kpv2}
C.~E. Kenig, G.~Ponce and L.~Vega,
\newblock A bilinear estimate with applications to the KdV equation,
\newblock {\em J. Amer. Math. Soc.}  9, 573--603, 1996.

\bibitem{KPV5}
C.~E. Kenig, G.~Ponce, and L.~Vega,
\newblock On the ill-posedness of some canonical dispersive equations,
\newblock {\em Duke Math. J.} 106, 617--633, 2001.

\bibitem{koch-m}
H.~Koch and J.~L. Marzuola,
\newblock Small data scattering and soliton stability in $\dot{H}^{-\frac16}$ for the quartic KdV equation,
\newblock arXiv:1001.4747v2 [math.AP].

\bibitem{KdV95}
D.~Korteweg and G.~de~Vries,
\newblock On the change of form of long waves advancing in a rectangular canal, and on a new type of
long stationary waves,
\newblock {\em Philos. Mag.} 539, 422--443, 1895.

\bibitem{marmer}
Y. Martel and F. Merle,
\newblock Blow up in finite time and dynamics of blow up solutions for the $L^2$-critical generalized KdV equation,
\newblock {\em J. Amer. Math. Soc.}  15, 617--664, 2002.

\bibitem{mer}
F. Merle,
\newblock Existence of blow-up solutions in the energy space for the critical generalized KdV equation,
\newblock{ \em J. Amer. Math. Soc.}  14, 555--578, 2001.

\bibitem{miao}
C. Miao, S. Shao, Y. Wu, and  G. Xu,
\newblock The low regularity global solutions for the critical generalized KdV equation,
\newblock arXiv:0908.0782v3 [math.AP].

\bibitem{Na} B. V. Sz. Nagy, \"Uber Integralgleichungen zwischen einer Funktion und ihrer Ableitung,
\textit{Acta Sci. Math.} 10, 64--74, 1941.

\bibitem{tao} T. Tao, Scattering for the quartic generalised Korteweg-de Vries equation, J. Dif. Eq. 3,
623--651, 2006.

\bibitem{Tz}  N. Tzvetkov, Remark on the local ill-posedness for KdV equation, \textit{C. R. Acad. Sci. Paris}
 329, 1043--1047, 1999.

\bibitem{W83}  M. Weinstein, Nonlinear Schr\"odinger equations and sharp interpolation
estimates,  \textit{Comm. Math. Phys.} 87, 567--576, 1983.




\end{thebibliography}

\end{document}